\documentclass{mcom-l}
\usepackage{amsfonts}
\usepackage{bbm}
\usepackage{amsmath,booktabs,ctable,threeparttable}
\usepackage{amssymb,amsfonts,boxedminipage}
\usepackage{amsthm}
\usepackage{algorithm}
\usepackage{algorithmic}
\usepackage{epsfig,graphicx,picins,picinpar,subfigure}
\usepackage{multirow}
\usepackage{verbatim}
\usepackage{bm}

\numberwithin{equation}{section}
\newtheorem{thm}{Theorem}[section]
\newtheorem{cor}[thm]{Corollary}
\newtheorem{lem}[thm]{Lemma}

\newtheorem{prop}[thm]{Proposition}
\theoremstyle{definition}
\newtheorem{defn}[thm]{Definition}
\newtheorem{exmp}{Example}

\makeatletter

\newcommand{\abs}[1]{\left\vert#1\right\vert}
\newcommand{\set}[1]{\left\{#1\right\}}

\newcommand{\E}[1]{\mathbb{E}\left[#1\right]}
\newcommand{\R}[1]{\mathbb{R}#1}
\newcommand{\U}[1]{\mathbb{U}#1}

\newcommand{\var}[1]{\mathrm{Var}[#1]}

\newcommand{\I}[1]{\mathbb{I}\{#1\}}

\newcommand{\mrd}{\,\mathrm{d}}
\newcommand{\trans}{\,\mathrm{T}}

\begin{document}

\title[QMC for  discontinuous integrands with singularities]{Quasi-Monte Carlo for discontinuous integrands with singularities along the boundary of the unit cube}


\author{Zhijian He}
\address{Lingnan (University) College, Sun Yat-sen University, Guangzhou 510275, China}
\curraddr{}
\email{hezhijian87@gmail.com}
\thanks{This work was supported by the National Science Foundation of China under grant $71601189$.}


\date{}


\begin{abstract}
	This paper studies randomized quasi-Monte Carlo (QMC) sampling for discontinuous integrands having singularities along the boundary of the unit cube $[0,1]^d$. 
	Both discontinuities and singularities are extremely common in the pricing and hedging of financial derivatives and have a tremendous impact on the accuracy of QMC. It was previously known that the root mean square error of  randomized QMC is only $o(n^{−1/2})$ for
	discontinuous functions with singularities. We find that under some mild conditions, randomized QMC yields an expected error of $O(n^{-1/2-1/(4d-2)+\epsilon})$ for arbitrarily small $\epsilon>0$. Moreover, one can get a better rate if the boundary of discontinuities is parallel to some coordinate axes. As a by-product, we find that the expected error rate attains $O(n^{-1+\epsilon})$ if the discontinuities are QMC-friendly, in the sense that all the discontinuity boundaries are parallel to coordinate
	axes. The results can be used to assess the QMC accuracy for some typical problems from financial engineering.
\end{abstract}

\maketitle

\section{Introduction}
It is known that quasi-Monte Carlo (QMC) integration over the  unit cube $[0,1]^d$ yields an asymptotic error rate of $O(n^{-1}(\log n)^d)$ when the integrand has bounded variation in the sense of Hardy and Krause (BVHK); see \cite{nied:1992} for details. In this paper we consider integrands that are discontinuous and have singularities along the boundary of  the unit cube $[0,1]^d$. Such integrands cannot be BVHK because they are unbounded. Both discontinuities and singularities are extremely common in computational finance. Specifically,
many problems arising from option pricing can be formulated as an integral over an unbounded domain $\mathbb{R}^d$ (see Glasserman \cite{glas:2004} and references therein). A necessary first step in applying QMC methods to a practical integral formulated over $\mathbb{R}^d$ is
to transform the integral into the unit cube $[0, 1]^d$. The transformation may introduce singularities at the boundary. In addition, discontinuities appear in the pricing and hedging of financial derivatives (e.g., barrier options) and have a tremendous impact on the accuracy of QMC method \cite{he:wang:2014,wang:2013}.

Formally, we are interested in integrands of the form
\begin{equation}\label{eq:target}
f(\bm u)=g(\bm u)\I{\bm{u}\in \Omega},
\end{equation}
where $\Omega\subset [0,1]^d$ and $g$ has singularities along the unit cube $[0,1]^d$. The integrand $f$ has a singularity at the boundary if $\Omega\cap [0,1]^d\neq \emptyset$. The QMC estimate of the integral $$I(f)=\int_{[0,1]^d}f(\bm u)\mrd \bm u$$ is given by the average of $n$ samples
\begin{equation}\label{eq:estimate}
\hat{I}(f) = \frac 1 n \sum_{i=1}^n f(\bm u_i),
\end{equation}
with carefully chosen $\bm u_i\in [0,1]^d$. He and Wang \cite{he:wang:2015} studied the convergence rate of RQMC for discontinuous functions of the form \eqref{eq:target}, but $g$ is assumed to be BVHK that excludes singularities. Under some mild assumptions on $\Omega$,  they proved that the root mean square error  of RQMC is $O(n^{-1/2-1/(4d-2)+\epsilon})$ for any $\epsilon>0$. If some discontinuity boundaries are parallel to some coordinate axes, the rate can be further improved to $O(n^{-1/2-1/(4d_u-2)+\epsilon})$, where $d_u$ denotes the so-called irregular dimension, that is the number of axes which are not parallel to the discontinuity boundaries. The results in He and Wang \cite{he:wang:2015} cannot be applied to our setting because $g$ is not BVHK. 

Owen \cite{owen:2006} considered functions singular around any or all of the $2^d$ corners of $[0,1]^d$ and got some error rates that can be as good as $O(n^{-1+\epsilon})$ if the singular function obeys a strict enough growth rate. Owen \cite{owen:2006b} found the convergence rate of RQMC for integrands with point singularities with unknown locations. More recently, Basu and Owen \cite{basu:2016b} considered functions on the square $[0,1]^2$ that may are singular along a diagonal in the square. A key strategy in \cite{owen:2006}, \cite{owen:2006b} and \cite{basu:2016b} is to employ another function that has finite variation to approximate the singular function. The approximation has low variation. Motivated by these works, we use a low variation approximation $\tilde{g}$ to $g$, resulting in an approximation of $f$, given by 
\begin{equation}\label{eq:target2}
\tilde f(\bm u)=\tilde g(\bm u)\I{\bm{u}\in \Omega}.
\end{equation}
Then using triangle inequality gives
\begin{equation*}
\abs{I(f)-\hat I(f)}\leq \abs{I(f)-I(\tilde f)}+\abs{I(\tilde f)-\hat I(\tilde f)}+\abs{\hat I(\tilde f)-\hat I(f)}.
\end{equation*}
Suppose that $\bm u_1,\dots,\bm u_n$ in \eqref{eq:estimate} are RQMC points where each $\bm u_i\sim \U{([0,1]^d)}$ individually; then 
\begin{equation*}
\E{\abs{\hat I(\tilde f)-\hat I(f)}}\leq \frac 1 n \sum_{i=1}^n \E{\abs{f(\bm u_i)-\tilde f(\bm u_i)}}=I{(|f-\tilde f|)}.
\end{equation*}
As a result,
\begin{equation*}
\E{|I(f)-\hat I(f)|}\leq 2I{(|f-\tilde f|)}+\E{|I(\tilde f)-\hat I(\tilde f)|}.
\end{equation*}
To get an expected error bound, it suffices to bound the approximation error $I{(|f-\tilde f|)}$ and the RQMC integration error for the function \eqref{eq:target2}. An upper bound of the approximation error can be obtained similarly as in Owen \cite{owen:2006}, which requires a growth condition on $g$. For the later, we will follow the analysis in \cite{he:wang:2015} since $\tilde{g}$ is BVHK.

This paper finds some rates of convergence for RQMC integration of the function \eqref{eq:target}. Suppose that $g$ obeys a strict enough growth rate. We find that the expected error in RQMC is $O(n^{-1/2-1/(4d-2)+\epsilon})$.  Moreover, one can get a better rate $O(n^{-1/2-1/(4d_u-2)+\epsilon})$ if the boundary of $\Omega$ is parallel to some coordinate axes. These results are similar to those  found in  He and Wang \cite{he:wang:2015}. As a by-product, the expected error rate attains $O(n^{-1+\epsilon})$ if the discontinuities involved in \eqref{eq:target} are QMC-friendly (they are parallel to coordinate axes as discussed in \cite{Wang2011}). Our theoretical results can explain why QMC integration can be effective for some problems with both discontinuities and singularities in financial engineering. They also reveal the effects of discontinuities and singularities on  QMC accuracy.

This paper is organized as follows. Section \ref{sec:preliminary} gives the background on $(t,m,d)$-nets, $(t,d)$-sequences and the randomization technique proposed by \cite{Owen1995}. The singular function $g$ is assumed to satisfy the growth condition. Some results from \cite{he:wang:2015} are also reviewed.
Convergence results of the expected error in RQMC for the function \eqref{eq:target} are formally stated and proved in Section \ref{sec:rates}. Section~\ref{sec:examples} presents some examples arising from computational finance in which the growth condition is satisfied with arbitrarily small positive rates. Section \ref{sec:concl} concludes this paper.

\section{Background}\label{sec:preliminary}

\subsection{Digital nets and sequences}  	
Throughout this paper, we work with scrambled nets and sequences following the framework of He and Wang \cite{he:wang:2015}. The integer $b\geq 2$  serves as a base. To begin with, we define an  elementary interval in base $b$.

\begin{defn}
	An elementary interval in base $b$ is a subset of $[0,1)^d$ of the form
	\begin{equation}\label{eq:EI}
	E = \prod_{i=1}^d\bigg[\frac{t_i}{b^{k_i}},\frac{t_i+1}{b^{k_i}}\bigg),
	\end{equation}
	where $k_i,t_i\in \mathbb{N}$ with $t_i<b^{k_i}$ for $i=1,\dots,d$.
\end{defn}

\begin{defn}\label{defnnets}
	Let $t$ and $m$ be nonnegative integers with $t\leq m$. A finite sequence $\bm{u}_1,\dots,\bm{u}_{b^m} \in [0,1)^d$ is a $(t,m,d)$-net in base $b$ if every elementary interval in base $b$ of volume $b^{t-m}$ contains exactly $b^t$ points of the sequence.
\end{defn}
\begin{defn}
	Let $t$ be a nonnegative integer. An infinite sequence $(\bm u_i)_{i\ge 1}$ with $\bm{u}_i\in [0,1)^d$ is a $(t,d)$-sequence in base $b$ if for all $k\geq 0$ and $m\geq t$ the finite sequence $\bm{u}_{kb^m+1},\dots,\bm{u}_{(k+1)b^m}$ is a $(t,m,d)$-net in base $b$.
\end{defn}
\subsection{Scrambling}
Owen \cite{Owen1995} applied a scrambling scheme on the nets that retains the net property. Let $\bm{u}_1,\dots,\bm{u}_n$ be a $(t,m,d)$-net or the first $n$ elements of a $(t,d)$-sequence in base $b$ where $\bm{u}_i=(u_i^1,\dots,u_i^d)$. We may write the components of $\bm{u}_i$ in their base $b$ expansion $u_i^j = \sum_{k=1}^{\infty} a_{ijk}b^{-k},$
where $a_{ijk}\in\{0,\dots,b-1\}$ for all $i,j,k$. The scrambled version of $\bm{u}_1,\dots,\bm{u}_n$ is a sequence $\tilde{\bm{u}}_1,\dots,\tilde{\bm{u}}_n$ with  $\tilde{\bm{u}}_i=(\tilde{u}_i^1,\dots,\tilde{u}_i^d)$ written as $\tilde{u}_i^j = \sum_{k=1}^{\infty}\tilde{a}_{ijk}b^{-k},$ where $\tilde{a}_{ijk}$ are defined in terms of random permutations of the $a_{ijk}$. The permutation applied to $a_{ijk}$ depends on the values of $a_{ijh}$ for $h=1,\dots,k-1$. Specifically, $\tilde{a}_{ij1} = \pi_j(a_{ij1}),\ \tilde{a}_{ij2} = \pi_{ja_{ij1}}(a_{ij2}),\ \tilde{a}_{ij3} = \pi_{ja_{ij1}a_{ij2}}(a_{ij3})$, and in general
$$
\tilde{a}_{ijk} = \pi_{ja_{ij1}a_{ij2}\dots  a_{ijk-1} }(a_{ijk}).
$$
Each permutation $\pi_\bullet$ is uniformly distributed over the $b!$ permutations of $\{0,\dots,b-1\}$, and the permutations are mutually independent.

\subsection{Convergence results from He and Wang \cite{he:wang:2015}}
He and Wang \cite{he:wang:2015} considered integrands of the form $f(\bm u)=g(\bm u)\I{\bm{u}\in \Omega}$, where $g$ is BVHK and the boundary of $\Omega$ admits a $(d-1)$-dimensional Minkowski content defined below. 
\begin{defn}
	For a set $\Omega\subset[0,1]^d$, define
	\begin{equation}
	\mathcal{M}(\partial \Omega)=\lim_{\epsilon\downarrow 0}\frac{\lambda_d((\partial \Omega)_\epsilon)}{2\epsilon},
	\end{equation}
	where $(A)_\epsilon:=\{x+y|x\in A,\Vert y\Vert\leq \epsilon\}$, and $\Vert\cdot\Vert$ denotes the usual Euclidean norm.
	If $\mathcal{M}(\partial \Omega)$ exists and finite, then $\partial \Omega$ is said to admit a $(d-1)$-dimensional Minkowski content.
\end{defn}

In the terminology of geometry, $\mathcal{M}(\partial \Omega)$  is known as the surface area of the set $\Omega$. The Minkowski content has a clear intuitive basis, compared to the Hausdorff measure \cite{matt:1995} that provides an alternative to quantify the surface area. We should note that the Minkowski content coincides with the Hausdorff measure, up to a constant factor, in regular cases.  It is known that the boundary of any convex set in $[0,1]^d$ has a $(d-1)$-dimensional Minkowski content. In this case, $\mathcal{M}(\partial \Omega)\leq 2d$ since the surface area of a convex set  in $[0,1]^d$ is bounded by the surface area of  the unit cube $[0,1]^d$, which is $2d$. More generally, Ambrosio et al. \cite{Ambr:2008} found that if $\Omega$ has a Lipschitz boundary, then $\partial \Omega$ admits a ($d-1$)-dimensional Minkowski content. In their terminology, a set $\Omega$ is said to have a Lipschitz boundary if for every boundary point $a$ there exists a neighborhood $A$ of $a$, a rotation $R$ in $\R^d$ and a Lipschitz function $f: \R^{d-1}\to \R$ such that $R(\Omega\cap A)=\set{(x,y)\in (\R^{d-1}\times \R)\cap R(A)|y\geq f(x)}$. In other words, $\Omega\cap A$ is the epigraph of a Lipschitz function.

He and Wang \cite{he:wang:2015} showed that a faster convergence rate of RQMC can be achieved if the set $\Omega$ has some regularity. They studied  partially axis-parallel sets as defined below. For a positive integer $d$, denote $1{:}d=\{1,\dots,d\}$. For a set $u\subset{1{:}d}$, denote the cardinality of
$u$ as $\abs{u}$ and $-u=1{:}d\backslash u$.

\begin{defn}
	A set $\Omega$ is said to be a partially axis-parallel set with  irregular dimension $d_u=\abs{u}$ if
	\begin{equation}\label{eq:clinder}
	\Omega=\Omega_u\times\prod_{i\notin u}[a_i,b_i),
	\end{equation}
	where $u\subset 1{:}d$, $d_u<d$, $0\leq a_i<b_i\leq 1$ for $i\notin u$, and $\Omega_u$ is a Lebesgue measurable subset of $\prod_{i\in u}[0,1)$. The quantity $d_u$ counts the number of axes which are not parallel to the boundaries of $\Omega$. 
\end{defn}

Denote $V_{\mathrm{HK}}(g)$ as the variation of the function $g$ in the sense of Hardy and  Krause. See \cite{owen:2005} for an outline of the variation.
The following proposition summarizes the convergence results found in He and Wang \cite{he:wang:2015}. 

\begin{prop}\label{lem:rates}
	Suppose that $f(\bm u)=g(\bm u)\I{\bm{u}\in \Omega}$, where $g\in L^2[0,1]^d$ satisfies $V_{\mathrm{HK}}(g) < \infty$. Assume that the sequence $\bm u_1,\dots, \bm u_n$ in \eqref{eq:estimate} is a
	scrambled $(t,d)$-sequence in base $b\geq 2$. If $\partial \Omega$ admits a $(d-1)$-dimensional Minkowski content, then for all sufficiently large $n$,
	\begin{equation}\label{eq:rate1}
	\var{\hat I(f)} \le c_{d,\Omega}M_g^2 n^{-1-1/(2d-1)}(\log n)^{2d/(2d-1)},
	\end{equation}
	where $c_{d,\Omega}$ depends only on $\Omega$ and $d$, and
	\begin{equation}\label{eq:cg}
	M_g= \max\left(V_{\mathrm{HK}}(g),\sup_{\bm u\in [0,1]^d}\abs{g(\bm u)}\right).
	\end{equation}
	If $\Omega$ is a partially axis-parallel set with irregular dimension $d_u$ defined by \eqref{eq:clinder}, where $\partial \Omega_u$ admits a $(d_u-1)$-dimensional Minkowski content, then for all sufficiently large $n$,
	\begin{equation}\label{eq:rate2}
	\var{\hat I(f)} \le c_{d,\Omega}M_g^2 n^{-1-1/(2d_u-1)}(\log n)^{2d/(2d_u-1)}.
	\end{equation}
\end{prop} 
\begin{proof}
	The first part \eqref{eq:rate1} has been proven in Theorem 3.5 of \cite{he:wang:2015}, and the second part \eqref{eq:rate2} has been proven in Theorem 3.6 of \cite{he:wang:2015}. Here we show the implicit constants  in the upper bounds of the scrambled net variances because they are useful in the following proofs.
\end{proof}

Functions of BVHK must necessarily be bounded. So $M_g$ given by \eqref{eq:cg} is finite. However, for many practical problems arising from computational finance, $g$ has singularities on the surface of the unit cube $[0,1]^d$. In this cases, $g$ is unbounded, and hence $g$ has infinite variation. The conditions in Proposition~\ref{lem:rates} are thus not satisfied. Before establishing the convergence rate of scrambled net errors for singular integrands, we suppose that $g$ satisfies the growth condition as studied in Owen \cite{owen:2006}.

\subsection{Growth Condition}
For a set $v\subseteq 1{:}d$, $\partial^{v}g$ denotes the mixed partial derivative
of $g$ taken once with respect to components with indices in $v$. Following Owen \cite{owen:2006}, we first introduce a growth condition for $g$ on $(0,1)^d$ that may become singular at the boundary of $[0,1]^d$ as shown in some integrands in the valuation of options with unbounded payoffs (see Section~\ref{sec:examples} for some examples). 
\begin{defn}
	A function $g$ defined on $(0,1)^d$ is said to satisfy the boundary growth condition if
	\begin{equation}\label{eq:grow}
	\abs{\partial^{v}g(\bm u)}\leq B\prod_{i\in v}\min(u_i,1-u_i)^{-A_i-1}\prod_{i\notin v}\min(u_i,1-u_i)^{-A_i}
	\end{equation}
	holds for some $A_i>0$, some $B<\infty$ and all $v\subseteq 1{:}d$.	
\end{defn}
The boundary growth condition is the second growth condition described in Owen \cite{owen:2006}. Owen \cite{owen:2006b} and Basu and Owen \cite{basu:2016b} studied other types of growth conditions for point singularities and singularities along a diagonal in the square, respectively. 
Large values of $A_i$ correspond to more severe singularities. When $\max_i A_i\geq 1$ the upper bound for $\abs{g}$ is not even integrable.  When $\max_i A_i< 1/2$, then $f^2$ is integrable and Monte Carlo sampling has a root mean square error of $O(n^{-1/2})$. We use a region to avoid the singularities as
\begin{equation}
K(\epsilon) = \{\bm u\in[0,1]^d|\prod_{1\le i\le d}\min(u_i,1-u_i)\ge \epsilon\},
\end{equation}
for small $\epsilon>0$. We now define an extension $g_\epsilon$ of $g$ from $K(\epsilon)$ to $[0,1]^d$ such that $g_\epsilon(\bm u)=g(\bm u)$ for $\bm u\in K(\epsilon)$.
\begin{defn}
A set $K\subset [0,1]^d$ is said to be Sobol' extensible with anchor $\bm c$ if for every $\bm u\in K$ the rectangle $\prod_{i=1}^d[\min(u_i,c_i),\max(u_i,c_i)]\subset K$. 	
\end{defn}
It is easy to see that $K(\epsilon)$ is Sobol' extensible with anchor $\bm c = (1/2,\dots,1/2)$. So one may write 
\begin{equation}
g(\bm u) = g(\bm{c}) +\sum_{v\neq \emptyset}\int_{[\bm c^v,\bm{u}^v]}\partial^{v}g(\bm z^v{:}\bm c^{-v})\mrd \bm z^v,
\end{equation}
and then the desired low variation approximation of $g$ is given by
\begin{equation}\label{eq:extg}
g_\epsilon(\bm u) = g(\bm{c})+ \sum_{v\neq \emptyset}\int_{[\bm c^v,\bm{u}^v]}\partial^{v}g(\bm z^v{:}\bm c^{-v})\I{\bm z^v{:}\bm c^{-v}\in K(\epsilon)}\mrd \bm z^v,
\end{equation}
where $\bm z^v{:}\bm c^{-v}$ denotes the point $\bm y\in[0,1]^d$ with $y_j=z_j$ for $j\in v$ and $y_j=c_j$ for  $j\notin v$.
\section{Expected errors in RQMC}\label{sec:rates}

\begin{prop}\label{prop:hk}
	If $g$ satisfies the boundary growth condition \eqref{eq:grow},  then for any $\eta>0$ there exists  $C_1<\infty$ such that
	\begin{equation}\label{eq:upHK}
	V_{\mathrm{HK}}(g_\epsilon)\leq C_1 \epsilon^{-\max_i A_i-\eta}.
	\end{equation}
	If there is a unique maximum among $A_1,\dots,A_d$, then \eqref{eq:upHK} holds with $\eta=0$.
\end{prop}
\begin{proof}
	See the proof of Theorem 5.5 in \cite{owen:2006}.
\end{proof}

\begin{prop}\label{prop:tranError}
	Let $f_\epsilon(\bm{u}) = g_\epsilon(\bm u)\I{\bm u\in \Omega}$, where $g_\epsilon$ is given by \eqref{eq:extg}. If $g$ satisfies the boundary growth condition \eqref{eq:grow} with $\max_i A_i <1$,  then for any $\eta\in (0,1-\max_i A_i)$, there exists  $C_2<\infty$ such that 
	\begin{equation}\label{eq:error}
	I(\abs{f-f_\epsilon})\leq C_2\epsilon^{1-\max_iA_i-\eta}.
	\end{equation} 
	If there is a unique maximum among $A_1,\dots,A_d$, then \eqref{eq:error} holds with $\eta=0$.
\end{prop}
\begin{proof}	
	From the proof of Theorem 5.5 in \cite{owen:2006} which is based on a result in \cite{sobol:1973}, we have $I(\abs{g-g_\epsilon})\leq C_2\epsilon^{1-\max_iA_i-\eta}$.	
	The upper bound \eqref{eq:error} then follows from $I(\abs{f-f_\epsilon})= I(\abs{g-g_\epsilon}\I{\bm{u}\in\Omega}) \le I(\abs{g-g_\epsilon})$.
\end{proof}
\begin{prop}\label{prop:max}
	If $g$ satisfies the boundary growth condition \eqref{eq:grow}, then for any $\eta>0$ there exists  $C_3<\infty$ such that
	\begin{equation}\label{eq:upg}
	\sup_{\bm u\in [0,1]^d}\abs{g_\epsilon(\bm u)}\leq C_3\epsilon^{-\max_i A_i-\eta}.
	\end{equation}
	If there is a unique maximum among $A_1,\dots,A_d$, then \eqref{eq:upg} holds with $\eta=0$.
\end{prop}
\begin{proof}
	The procedure is similar to the proof of Theorem 5.5 in \cite{owen:2006}.
	Combining \eqref{eq:extg} with the boundary growth condition \eqref{eq:grow}, we have 
	\begin{align}
	\abs{g_\epsilon(\bm u)} &\leq  \abs{g(\bm{c})}+ \sum_{v\neq \emptyset}\int_{[\bm c^v,\bm{u}^v]}\abs{\partial^{v}g(\bm z^v{:}\bm c^{-v})}\I{\bm z^v{:}\bm c^{-v}\in K(\epsilon)}\mrd \bm z^v\notag\\
	&\leq  \abs{g(\bm{c})}+B \sum_{v\neq \emptyset} I_v\prod_{i\notin  v}\min(c_i,1-c_i)^{-A_i},\label{eq:boundg}
	\end{align}
	where
	\begin{align*}
	I_v:=\int_{[\bm{0}^v,\bm{1}^v]}\prod_{i\in v}\min(z_i,1-z_i)^{-A_i-1}\I{\bm z^v{:}\bm c^{-v}\in K(\epsilon)}\mrd\bm z^v.
	\end{align*}
	We first assume that $A_1,\dots,A_d$ are distinct positive numbers. Let $m(v)=\arg\max_{i\in v}A_i$ and $\tilde{v}=v-\{m(v)\}$. Let $e(\bm z_{\tilde{v}})=\prod_{i\in \tilde{v}}\min(z_i,1-z_i)\prod_{i\notin  v}\min(c_i,1-c_i)$. Then
	\begin{align*}
	I_v&=\int_{[\bm{0}^{\tilde{v}},\bm{1}^{\tilde{v}}]}\prod_{i\in \tilde{v}}\min(z_i,1-z_i)^{-A_i-1}\left(\int_{\min(y,1-y)\geq  \epsilon/e(\bm z_{\tilde{v}})}\min(y,1-y)^{-A_{m(v)}-1}\mrd y\right)\mrd \bm z_{\tilde{v}}\\
	&=2\int_{[\bm{0}^{\tilde{v}},\bm{1}^{\tilde{v}}]}\prod_{i\in \tilde{v}}\min(z_i,1-z_i)^{-A_i-1}\left(\int_{\epsilon/ e(\bm z_{\tilde{v}})}^{1/2}y^{-A_{m(v)}-1}\mrd y\right)\mrd \bm z_{\tilde{v}}\\
	&\leq 2\int_{[\bm{0}^{\tilde{v}},\bm{1}^{\tilde{v}}]}\prod_{i\in \tilde{v}}\min(z_i,1-z_i)^{-A_i-1}\frac{( \epsilon/e(\bm z_{\tilde{v}}))^{-A_{m(v)}}}{A_{m(v)}}\mrd \bm z_{\tilde{v}}\\
	&=2\frac{\epsilon^{-A_{m(v)}}}{A_{m(v)}}\prod_{i\notin  v}\min(c_i,1-c_i)^{A_{m(v)}}\int_{[\bm{0}^{\tilde{v}},\bm{1}^{\tilde{v}}]}\prod_{i\in \tilde{v}}\min(z_i,1-z_i)^{A_{m(v)}-A_i-1}\mrd \bm z_{\tilde{v}}\\
	&\leq 2\frac{\epsilon^{-A_{m(v)}}}{A_{m(v)}}\prod_{i\notin  v}\min(c_i,1-c_i)^{A_{m(v)}}\prod_{i\in\tilde{v}}\frac{2}{A_{m(v)}-A_i}\\
	&= C_v\epsilon^{-A_{m(v)}},
	\end{align*}
	where $C_v$ is a finite constant. It then follows from \eqref{eq:boundg} that 
	\begin{equation}\label{eq:maxge}
	\abs{g_\epsilon(\bm u)} \leq  \abs{g(\bm{c})}+\tilde{B}\epsilon^{-\max_i A_i}
	\end{equation}
	for some finite $\tilde{B}$. 
	
	If $A_j=A_k<\max_iA_i$ for some $j\neq k$, then we increase some of the $A_i$ so that $A_1,\dots,A_d$ are distinct, while leaving $\max_iA_i$ unchanged. Then \eqref{eq:maxge} also holds if there is a unique maximum among $A_1,\dots,A_d$. We thus have \eqref{eq:upg} with $\eta=0$ due to $\abs{g(\bm{c})}<\infty$. If there are two or more maximums among $A_1,\dots,A_d$, then these maximums can be increased to distinct values, while raising $\max_iA_i$ by no more than $\eta$.
\end{proof}
\begin{thm}\label{thm:rates}
	Suppose that $f$ is given by \eqref{eq:target}, where $g$ satisfies the boundary growth condition \eqref{eq:grow}  with $\max_i A_i <1$. Assume that the sequence $\bm u_1,\dots, \bm u_n$ in \eqref{eq:estimate} is a
	scrambled $(t,d)$-sequence in base $b\geq 2$. If $\partial \Omega$ admits a $(d-1)$-dimensional Minkowski content, then for any $\eta\in (0,1-\max_i A_i)$,
	\begin{equation}\label{eq:errorbound1}
	\E{\abs{I(f)-\hat{I}(f)}} = O(n^{-\gamma(1/2+1/(4d-2))}(\log n)^{\gamma d/(2d-1)}),
	\end{equation}
	where $\gamma=1-\max_iA_i-\eta$.
	If $\Omega$ is a partially axis-parallel set with  irregular dimension $d_u$ defined by \eqref{eq:clinder}, where $\partial \Omega_u$ admits a $(d_u-1)$-dimensional Minkowski content, then
	\begin{equation}\label{eq:errorbound2}
	\E{\abs{I(f)-\hat{I}(f)}} = O(n^{-\gamma(1/2+1/(4d_u-2))}(\log n)^{\gamma d/(2d_u-1)}).
	\end{equation}
	If there is a unique maximum among $A_1,\dots,A_d$, then \eqref{eq:errorbound1} and \eqref{eq:errorbound2} hold with $\gamma=1-\max_iA_i$.
\end{thm} 
\begin{proof}
	Using the triangle inequality and  the unbiasedness of the estimate $\hat{I}(f_\epsilon)$, we have
	\begin{align*}
	\E{\abs{I(f)-\hat{I}(f)}}&=\E{\abs{I(f)-I(f_\epsilon)+I(f_\epsilon)-\hat{I}(f_\epsilon)+\hat{I}(f_\epsilon)-\hat{I}(f)}}\\
	&\leq I(\abs{f-f_\epsilon})+\E{\abs{I(f_\epsilon)-\hat{I}(f_\epsilon)}}+\E{\hat{I}(\abs{f_\epsilon-f})}\\
	&\le 2I(\abs{f-f_\epsilon})+\var{\hat{I}(f_\epsilon)}^{1/2}.
	\end{align*}
	Proposition~\ref{prop:tranError} gives $I(\abs{f-f_\epsilon})=O(\epsilon^\gamma)$, where $\gamma=1-\max_iA_i-\eta$. For the function $f_\epsilon(\bm u)=g_\epsilon(\bm u)\I{\bm{u}\in\Omega}$, it follows from Propositions~\ref{lem:rates}, \ref{prop:hk} and \ref{prop:max} that  
	\begin{align*}
	\var{\hat{I}(f_\epsilon)}^{1/2}&\leq \sqrt{c_{d,\omega}}M_g n^{-(1/2+1/(4d-2))}(\log n)^{d/(2d-1)}\\
	&=O(\epsilon^{-\max_i A_i-\eta}n^{-1/2-1/(4d-2)}(\log n)^{d/(2d-1)}).
	\end{align*}
	Consequently,
	\begin{align*}
	\E{\abs{I(f)-\hat{I}(f)}}=O(\epsilon^{\gamma})+O(\epsilon^{\gamma-1}n^{-1/2-1/(4d-2)}(\log n)^{d/(2d-1)}).
	\end{align*}
	Taking $\epsilon \propto n^{-1/2-1/(4d-2)}(\log n)^{d/(2d-1)}$ establishes \eqref{eq:errorbound1}. The rate  \eqref{eq:errorbound2} can be proved in the same way.
\end{proof}

From Theorem~\ref{thm:rates}, the rates for discontinuous integrands with singularities are not faster than those in Proposition~\ref{lem:rates}. 
RQMC is asymptotically superior to Monte Carlo when $A_i<1/(2d)$ for all $i$.
For some applications in computational finance (see Section~\ref{sec:examples} for some examples), it is possible that $g$ obeys the growth condition \eqref{eq:grow} with arbitrarily small positive $A_i$ for all $i$. The associated rates are presented in the following corollary, which are asymptotically superior to plain Monte Carlo sampling. In this case, the singularities may be regarded as QMC-friendly singularities because they deliver the best possible rate in our setting. 

\begin{cor}\label{cor:rates}
	Suppose that $f$ is given by \eqref{eq:target}, where $g$ satisfies the boundary growth condition with arbitrarily small positive $A_i$ for all $i$. Assume that the sequence $\bm u_1,\dots, \bm u_n$ in \eqref{eq:estimate} is a
	scrambled $(t,d)$-sequence in base $b\geq 2$. If $\partial \Omega$ admits a $(d-1)$-dimensional Minkowski content, then 
	\begin{equation}\label{eq:errorbound3}
	\E{\abs{I(f)-\hat{I}(f)}} = O(n^{-(1/2+1/(4d-2))+\epsilon})
	\end{equation}
	for arbitrary small $\epsilon>0$.
	If $\Omega$ is a partially axis-parallel set with  irregular dimension $d_u$ defined by \eqref{eq:clinder}, where $\partial \Omega_u$ admits a $(d_u-1)$-dimensional Minkowski content, then
	\begin{equation}\label{eq:errorbound4}
	\E{\abs{I(f)-\hat{I}(f)}} = O(n^{-(1/2+1/(4d_u-2))+\epsilon}).
	\end{equation}
\end{cor}

\section{Examples from computational finance}\label{sec:examples}
Let $S(t)$ denote the underlying price dynamics at time $t$ under the risk-neutral measure. In a simulation framework, it is common that the prices are simulated at discrete times $t_1,\dots,t_d$ satisfying $0=t_0<t_1<\dots<t_d=T$, where $T$ is the maturity of the financial derivative of interest. Without loss of generality,  we assume that the discrete times are evenly spaced, i.e., $t_i=i\Delta t$, where $\Delta t=T/d$. For simplicity, denote $S_i=S(t_i)$, and $\bm{S}=(S_1,\dots,S_d)^\top$. Under the risk-neutral measure, the price and the sensitivities of the financial derivative can 
be expressed as an expectation  $I=\E{f(\bm{S})}$ for a real function $f$ over $\mathbb{R}^d$. To translate the problem into QMC setting, we suppose that $S_i$ can be expressed as a function of $\bm u\sim \U{([0,1]^d)}$, denoted by $S_i(\bm u)$, after some appropriate transformations. Let $\bm{S}(\bm{u})=(S_1(\bm{u}),\dots,S_d(\bm{u}))^\top$. We thus have 
$$I=\E{f(\bm{S})}=\E{f(\bm{S}(\bm{u}))}=\int_{[0,1]^d}f(S_1(\bm{u}),\dots,S_d(\bm{u}))\mrd \bm u.$$
After the transformations, the integrand $f\circ \bm{S}$ is often unbounded
at the boundary of the unit cube. 

Many functions in the pricing and hedging of financial derivatives involve indicator functions, which can be expressed in the form
\begin{equation}\label{eq:payoff}
f(\bm{S}) = g(\bm S)\I{h(\bm S)\geq 0},
\end{equation}
where $g$ and $h$ are usually smooth functions over $\mathbb{R}^d$ (see \cite{he:wang:2014}). For pricing financial options, the factor $g$
determines the magnitude of the payoff and $h(\bm{S})>0$ gives the payout condition. For calculating Greeks by the pathwise method,
the target function often involves an indicator function as in \eqref{eq:payoff} even though the underlying payoff function is continuous. 

We assume that under the risk-neutral measure the asset follows the geometric Brownian motion
\begin{equation}
\frac{\mrd S(t)}{S(t)}=r\mrd t+\sigma \mrd B(t),\label{GBM}
\end{equation}
where $r$ is
the risk-free interest rate, $\sigma$ is the volatility and $B(t)$
is the standard Brownian motion. Under this assumption, the solution of \eqref{GBM} is analytically available
\begin{equation}
S(t)=S_0\exp[(r-\sigma^2/2)t+\sigma B(t)],\label{Solution}
\end{equation}
where $S_0$ is the initial price of the asset. Let $\bm{x}:=(B(t_1),\dots,B(t_d))^{\trans}$. We have $\bm{x}\sim N(\bm{0},\bm{\Sigma})$, where the entries of $\bm{\Sigma}$ are given by $\Sigma_{ij}=\Delta t\min(i,j)$.

Note that $\bm{\Sigma}$ is positive definite. Let $\bm{A}$ be a generating matrix satisfying $\bm{A}\bm{A}^\top=\bm{\Sigma}$. Let $\Phi$ be the cumulative distribution function of the standard normal distribution. Using the transformation $\bm{x}=\bm{A}\Phi^{-1}(\bm{u})$, it follows from
\eqref{Solution} that
\begin{align}\label{eq:si}
S_i(\bm{u})= S_0\exp\left[(r-\sigma^2/2)i\Delta t+\sigma \sum_{j=1}^da_{ij}\Phi^{-1}(u_j)\right].
\end{align}

To verify the boundary growth condition, we need  partial derivatives of ${g\circ\bm S}$ of order up to the dimension of the unit
cube. The multivariate
Faa di Bruno formula from \cite{cons:1996} gives an arbitrary mixed partial derivative of $g\circ \bm{S}$ in terms of partial derivatives of $g$ and $S_i$. Basu and Owen \cite{basu:2016a} also used the formula to  study the variation of some composition functions. The formula requires
that the needed derivatives exist. Let $\bm{\lambda}=(\lambda_1,\dots,\lambda_d)$ be a vector of nonnegative integers. Define $\abs{\bm \lambda}=\sum_{i=1}^{d}\lambda_i$. Denote $g_{\bm \lambda}$ as the derivative of $g$ taken $\lambda_i$ times with respect to the $i$th component.
It follows from Basu and Owen \cite{basu:2016a} that for $\emptyset \neq v \subseteq 1{:}d$,
\begin{equation}\label{eq:mulderi}
\partial^v(g\circ \bm{S})=\sum_{1\leq \abs{\bm \lambda}<\abs{v}} g_{\bm\lambda}(\bm{S})\sum_{s=1}^{\abs{v}}\sum_{(\ell_r,k_r)\in \widetilde{\mathrm{KL}}(s,v,\bm \lambda)}\prod_{r=1}^s \partial^{\ell_r}S_{k_r}(\bm{u}),
\end{equation}
where 
\begin{align*}
\widetilde{\mathrm{KL}}(s,v,\bm \lambda)=\{&(\ell_r,k_r)|r\in 1{:}s,\emptyset\neq \ell_r\subseteq 1{:}d,k_r\in1{:}d,\cup_{r=1}^s\ell_r =v,\\
&\ell_r\cap\ell_{r'}=\emptyset,\text{ for } r\neq r'\text{ and }\abs{\{j\in 1{:}d|k_j=i\}}=\lambda_i\}.
\end{align*}

The following lemma is a result of Owen \cite{owen:2006}. We prove it here also.
\begin{lem}\label{lem:partials}
	Suppose that $S_i$ is given by \eqref{eq:si}; then for any $v\subseteq 1{:}d$ and $i\in1{:}d$, 
	\begin{equation}\label{eq:partialS}
	\abs{\partial^{v}S_{i}(\bm{u})} \leq C_{i}\prod_{j\in v}\min(u_j,1-u_j)^{-A_j-1}\prod_{j\notin v}\min(u_j,1-u_j)^{-A_j}
	\end{equation}
	holds for arbitrarily small $A_j>0$ and $C_i<\infty$.
\end{lem}
\begin{proof}
	It follows from \eqref{eq:si} that
	\begin{equation*}
	\partial^{v}S_{i}(\bm{u}) = S_0\exp\left[(r-\sigma^2/2)i\Delta t+\sigma \sum_{j=1}^da_{ij}\Phi^{-1}(u_j)\right]\prod_{j\in v}\left(\sigma a_{ij}\frac{\mrd \Phi^{-1}(u_j)}{\mrd u_j}\right).
	\end{equation*}
	Note that $\Phi^{-1}(\epsilon)=-\sqrt{-2\log(\epsilon)}+o(1)$ and $\Phi^{-1}(1-\epsilon)=\sqrt{-2\log(\epsilon)}+o(1)$ as $\epsilon\downarrow 0$. This leads to $\exp(a\Phi^{-1}(u_j))=O(\min(u_j,1-u_j)^{-A_j/2})$ for any $A_j>0$ and an arbitrary $a\in \mathbb{R}$. Denote $\phi(x)=(2\pi)^{-1/2}\exp(-x^2/2)$ as the probability density of the standard normal distribution. We find that
	\begin{align*}
	\frac{\mrd \Phi^{-1}(u_j)}{\mrd u_j}&=\frac{1}{\phi(\Phi^{-1}(u_j))}\\&=\sqrt{2\pi}\exp\left[(\sqrt{-2\log(u_j)}+o(1))^2/2\right]\\&=O(\min(u_j,1-u_j)^{-1-A_j/2})
	\end{align*}
	for any $A_j>0$. The inequality \eqref{eq:partialS} is thus obtained.
\end{proof}

Since the function $S_{i}$ admits the boundary growth condition for arbitrarily small $A_j>0$, Owen
\cite{owen:2006} showed that  $S_{i}$ can be integrated with error $O(n^{-1+\epsilon})$ by the Halton sequence. However the results of Owen \cite{owen:2006} cannot be applied to our target function \eqref{eq:payoff} because it is discontinuous. To apply Theorem~\ref{thm:rates}, we need to verify the boundary growth condition for the composition $g\circ \bm S$. Combining  \eqref{eq:mulderi} and \eqref{eq:partialS}, we have
\begin{equation}\label{eq:boundgs}
\abs{\partial^v(g\circ \bm{S})}\leq B_1\sum_{1\leq \abs{\bm \lambda}\leq\abs{v}} \abs{g_{\bm\lambda}(\bm{S})}\prod_{j\in v}\min(u_j,1-u_j)^{-A_j-1}\prod_{j\notin v}\min(u_j,1-u_j)^{-A_j}
\end{equation}
for some finite $B_1$, arbitrarily small $A_j>0$ and $\emptyset \neq v \subseteq 1{:}d$. Therefore, the function $g\circ \bm S$ satisfies the growth condition \eqref{eq:grow} as long as 
\begin{equation}\label{eq:glambda}
\abs{g_{\bm\lambda}(\bm{S})}\leq B_2\prod_{j=1}^d\min(u_j,1-u_j)^{-\tilde{A}_j}
\end{equation}
holds for all $\abs{\bm \lambda}\leq\abs{v}$, $\tilde{A}_j>0$ and $B_2<\infty$. This may be verified for a broad range of functions since \eqref{eq:partialS} admits that
\begin{equation}\label{eq:boundedSi}
S_{i}(\bm{u}) \leq C_{i}\prod_{j=1}^d\min(u_j,1-u_j)^{-A_j}
\end{equation}
holds for arbitrarily small $A_j>0$.
In our applications, $g$ is rather simple so that $g_{\bm\lambda}$ is available. As illustrative examples, we next show that the growth condition \eqref{eq:grow} can be satisfied with 
arbitrarily small growth rates.

\begin{exmp}\label{example1}  
	The discounted payoff of an arithmetic Asian option is
	\begin{equation}\label{eq:asianpayoff}
	f(\bm{S})= e^{-rT}\left(S_A-K\right)\I{S_A>K},  
	\end{equation}
	where $S_A=(1/d)\sum_{j=1}^{d} S_j$ and $K$ is the strike price.
\end{exmp}
\begin{exmp}\label{example2} 	
	The pathwise estimate of the \textit{delta} of an arithmetic  Asian option   is
	\begin{equation}\label{asian delta}
	f(\bm{S}) =e^{-rT} \frac{S_A}{S_0} \I{S_A>K}. 
	\end{equation}
	The  \textit{delta} of an  option  is the sensitivity with respect   to the initial price $S_0$ of the underlying asset.
\end{exmp}
\begin{exmp}\label{example3} 
	An estimate of the \textit{gamma} of an arithmetic Asian option is
	\begin{equation}\label{eq:asiangamma}
	f(\bm{S}) =e^{-rT}  \frac{S_A \left( \log (S(t_1)/S_0)-(r+\sigma^2/2 )\Delta t\right) }{S_0^2\sigma^2 \Delta t  }  \I{S_A>K},
	\end{equation}
	which results from applying the pathwise method first and then the likelihood ration method (see \cite{glas:2004}). 	
	The \textit{gamma} is the second derivative with respect to the initial price $S_0$ of the underlying asset. 
\end{exmp}
\begin{exmp}\label{example4} 
	The pathwise estimate of the \textit{rho} of an arithmetic Asian option is
	\begin{equation}\label{eq:asianrho}
	f(\bm{S}) = e^{-rT}\left[\frac{\mrd S_A}{\mrd r} -T(S_A-K)\right]
	\I{S_A>K},
	\end{equation}
	where
	\begin{equation*}
	\frac{\mrd S_A}{\mrd r} =\frac{T}{d^2} \left( \sum_{j=1}^{d} j S(t_j)\right).
	\end{equation*}
	The  \textit{rho} of an   option  is the sensitivity with respect   to the risk-free interest rate  $r$.
	
\end{exmp}
\begin{exmp}\label{example5} 
	The pathwise estimate of the \textit{theta} of an arithmetic Asian option   is
	\begin{equation}\label{eq:asiantheta}
	f(\bm{S}) = e^{-rT} \left[\frac{d S_A}{d T}-r(S_A-K) \right]
	\I{S_A>K},
	\end{equation}
	where
	\begin{equation*}
	\frac{\mrd S_A}{\mrd T} = \frac{1}{d} \sum_{j=1}^{d}  S(t_j)\left[  \frac{\omega j }{2d}  + \frac{\log (S(t_j)/S_0)}{2T}  \right].
	\end{equation*}
	The \textit{theta} of an  option  is the sensitivity with respect   to the maturity of the option  $T$.	
\end{exmp}
\begin{exmp}\label{example6} 
	The pathwise estimate of the \textit{vega} of an arithmetic Asian option is
	\begin{equation}\label{eq:asianvega}
	f(\bm{S}) =e^{-rT} \frac{1}{d} \sum_{i=1}^{d} \frac{\mrd S(t_i)}{\mrd\sigma} \I{S_A>K}, 
	\end{equation}
	in which
	\begin{equation*}
	\frac{\mrd S(t_i)}{\mrd\sigma}=S(t_i)\frac{1}{\sigma}\left[ \log\left( \frac{S(t_i)}{S_0}\right) -\left( r+ \frac{1}{2}\sigma^2 \right) t_i   \right]. 
	\end{equation*}
	The  \textit{vega} of an  option  is the sensitivity with respect   to the volatility $\sigma$.
\end{exmp}	

\begin{thm}\label{thm:exm}
	Suppose that $f$ is one of the functions \eqref{eq:asianpayoff}--\eqref{eq:asianvega}, where $S_i$ is given by \eqref{eq:si}. Letting $f$ be expressed as the form \eqref{eq:payoff}, then $g\circ \bm S$ satisfies the boundary growth condition \eqref{eq:grow} with arbitrarily small $A_i>0$ for all $i$.
\end{thm}
\begin{proof}
	For the functions \eqref{eq:asianpayoff}--\eqref{eq:asianvega}, $g(\bm S)$ is a linear combination of some components $S_i$ and $\log(S_{i})S_{i'}$ for $i,i'\in 1{:}d$. It suffices to verify that these components satisfy \eqref{eq:glambda} because the linear combination then also satisfies \eqref{eq:glambda}. 
	
	Consider $g(\bm S)=S_i$ for any $i\in 1{:}d$. We have $\abs{g_{\bm\lambda}(\bm{S})}\leq 1$ for any $1\leq \abs{\bm \lambda}\leq\abs{v}$. For $\abs{\bm \lambda}=0$, $\abs{g_{\bm\lambda}(\bm{S})}=\abs{g(\bm S)}=S_i=O(\prod_{j=1}^d\min(u_j,1-u_j)^{-\tilde{A}_j})$ due to \eqref{eq:partialS}, for arbitrarily small $\tilde{A}_j>0$. In this case, $g_{\bm\lambda}$ satisfies \eqref{eq:glambda} with arbitrarily small growth rates.
	
	Consider $g(\bm S)=\log(S_i)S_{i'}$ for any $i\neq i'$. We have $g_{\bm\lambda}(\bm{S})=0$ if $\lambda_k>0$ for some $k\notin\{i,i'\}$ or $\lambda_{i'}>1$. So it reduces to consider three  cases: 
	\begin{enumerate}
		\item $1\leq \lambda_{i} \leq \abs{v},\ \lambda_k=0$ for any $k\neq i$;
		\item $0\leq \lambda_{i} \leq\abs{v}-1,\ \lambda_{i'}=1,\ \lambda_k=0$ for $k\notin\{i,i'\}$; and
		\item all $\lambda_k=0$.
	\end{enumerate}
For Case (1), we have
	\begin{equation*}
	g_{\bm\lambda}(\bm{S})=S_{i'}\frac{\mrd^{\lambda_i}\log(S_i) }{\mrd S_i^{\lambda_i}}=(-1)^{\lambda_i+1}c(\lambda_i)S_{i'}S_i^{-\lambda_i},
	\end{equation*}
	where $c(1)=1$ and $c(\lambda_i)=(\lambda_i-1)!$ for $\lambda_i>1$.
	For Case (2), we have
	\begin{equation*}
	g_{\bm\lambda}(\bm{S})=
	\begin{cases}
	(-1)^{\lambda_i+1}c(\lambda_i)S_i^{-\lambda_i},\ &\lambda_i>0\\
	\log(S_i),\ &\lambda_i=0.
	\end{cases}
	\end{equation*}
	For Case (3), $$g_{\bm\lambda}(\bm{S})=g(\bm S)=\log(S_i)S_{i'}.$$ From the proof of Lemma~\ref{lem:partials}, we have $$S_i^{-\lambda_i}=O\left(\prod_{j=1}^d\min(u_j,1-u_j)^{-\tilde{A}_j}\right)$$ and $$\abs{\log(S_i)}=O\left(\prod_{j=1}^d\min(u_j,1-u_j)^{-\tilde{A}_j}\right)$$ for arbitrarily small $\tilde{A}_j>0$. So $g_{\bm\lambda}$ satisfies \eqref{eq:glambda} with arbitrarily small growth rates.
	
	Consider $g(\bm S)=\log(S_i)S_i$ for any $i\in1{:}d$. If $\lambda_k=0$ for all $k\neq i$, we have
	\begin{equation*}
	g_{\bm\lambda}(\bm{S})=\frac{\mrd^{\lambda_i}(\log(S_i)S_i) }{\mrd S_i^{\lambda_i}}=\begin{cases}
	\log(S_i)S_i,\ &\lambda_i=0\\
	1+\log(S_i),\ &\lambda_i=1\\
	(-1)^{\lambda_i}c(\lambda_i-1)S_i^{-\lambda_i+1},\ &\lambda_i>1.
	\end{cases}
	\end{equation*}
	If $\lambda_k>0$ for some $k\neq i$, then $g_{\bm\lambda}(\bm{S})=0$. In this case, $g_{\bm\lambda}$ satisfies \eqref{eq:glambda} with arbitrarily small growth rates.
	
	Based on the reasoning above, it follows from \eqref{eq:boundgs} that for the functions \eqref{eq:asianpayoff}--\eqref{eq:asianvega}, $g\circ \bm S$ satisfies the boundary growth condition with arbitrarily small growth rates.
\end{proof}

Note that the statement in Theorem~\ref{thm:exm} holds for any decomposition of $\bm \Sigma=\bm{A}\bm{A}^\top$. To handle discontinuities, Wang and Tan \cite{wang:2013} proposed 
the orthogonal transformation (OT) method to make the discontinuities QMC
friendly, in the sense that all the discontinuity boundaries are parallel to coordinate
axes. The OT method delivers a special matrix $\bm A$ satisfying  $\bm{A}\bm{A}^\top=\bm \Sigma$ to generate the path \eqref{eq:si}. To illustrate its effects, let us consider the function 
\begin{equation}\label{eq:SG}
\tilde{f}(\bm S)=g(\bm S)\I{S_G>K},
\end{equation}
where $S_G=\prod_{i=1}^{d}S_i^{1/d}$ is the geometric average of the prices. For this function, applying the OT method arrives at
\begin{equation*}
\I{S_G>K}=\I{u_1>\kappa}
\end{equation*} 
for some constant $\kappa$ (see \cite{wang:2013} for determining the matrix $\bm A$). As a result, discontinuities occur only
on the axis-parallel hyperplane $u_1=\kappa$, which are  QMC-friendly. The function \eqref{eq:SG} is then transformed to $g(\bm S(\bm{u}))\I{\bm u\in\Omega}$, where  $\Omega=\{\bm u\in[0,1]^d|u_1>\kappa\}$. Note that the irregular dimension $d_u$ of the set $\Omega$ is one. Corollary~\ref{cor:rates} admits that the expected error rate of RQMC for the transformed function $g(\bm S(\bm{u}))\I{\bm u\in\Omega}$ is $O(n^{-1+\epsilon})$ if $g \circ \bm S$ has the same kind of singularities as examined in Theorem~\ref{thm:exm}. This suggests that making the discontinuities of the function \eqref{eq:SG} QMC friendly by the OT method can improve the efficiency of QMC greatly. For the functions of the form $g(\bm S(\bm{u}))\I{S_A>K}$ in the examples above, Wang and Tan
\cite{wang:2013} suggested that using the obtained matrix $\bm A$ for the function \eqref{eq:SG} can still be effective since $S_G$ is a good substitute for $S_A$. The usefulness of this strategy was illustrated  by several numerical examples in \cite{wang:2013,he:wang:2014}.

\section{Conclusion}\label{sec:concl}
We find that for discontinuous functions with singularities along the boundary of the unit cube $[0,1]^d$, RQMC has an expected error of $O(n^{-\gamma(1/2+1/(4d-2))+\epsilon})$ for $\gamma=1-\max_i A_i\in (0,1)$ depending on the growth rates $A_i$. The convergence rate $O(n^{-\gamma(1/2+1/(4d-2))+\epsilon})$ is a bit disappointing for large values of $A_i$. However,  the error rate can be as good as $O(n^{-(1/2+1/(4d-2))+\epsilon})$ for some problems from computational finance in which the growth rates are arbitrarily small. In these cases, it seems that the singularities have insignificant impact on QMC accuracy, compared to the rate for discontinuous integrands (without singularities) found in He and Wang \cite{he:wang:2015}. We also show theoretically the benefits of making discontinuities QMC-friendly, which have been shown empirically in various numerical examples of Wang and Sloan \cite{Wang2011} and Wang and Tan \cite{wang:2013}.

For singular functions (even discontinuous) satisfying the growth condition with arbitrarily small growth rates, QMC can lead to improved accuracy. It would be interesting to know how generally the problems from financial engineering fit into this setting, beyond those under the Gaussian model discussed in Section~\ref{sec:examples}.

\section*{Acknowledgments} 
The author thanks two anonymous referees for helpful suggestions  on improving this paper. The author also thanks Professor Art B. Owen and Kinjal Basu for sharing their work \cite{basu:2016b} with him. 

\providecommand{\bysame}{\leavevmode\hbox to3em{\hrulefill}\thinspace}
\providecommand{\MR}{\relax\ifhmode\unskip\space\fi MR }
\providecommand{\MRhref}[2]{%
	\href{http://www.ams.org/mathscinet-getitem?mr=#1}{#2}
}
\providecommand{\href}[2]{#2}

\end{document}